\documentclass[leqno,12pt]{article}

\usepackage{amstext    }
\usepackage{amsthm    }
\usepackage{a4}
\usepackage[mathscr]{eucal}
\usepackage{mathrsfs}

\usepackage{amsmath}
\usepackage{amssymb}
\usepackage{amscd}

\numberwithin{equation}{section}

\newtheorem{theorem}{Theorem}[section]

\newtheorem{proposition}[theorem]{Proposition}

\newtheorem{lemma}[theorem]{Lemma}

\newcommand{\cali}[1]{\mathscr{#1}}

\newcommand{\volume}{{\rm vol}}

\renewcommand{\Im}{\mathop{\mathrm{Im}}}
\renewcommand{\Re}{\mathop{\mathrm{Re}}}
\newcommand{\const}{\mathop{\mathrm{const}}}
\newcommand{\length}{\mathop{\mathrm{length}}}

\newcommand{\Ai}{\mathop{\mathrm{Ai}}\nolimits}
\newcommand{\comp}{{\mathop{\mathrm{comp}}\nolimits}}

\newcommand{\ind}{{\bf 1}}

\newcommand{\Cc}{\cali{C}}

\newcommand{\B}{\mathbb{B}}
\newcommand{\D}{\mathbb{D}}
\newcommand{\C}{\mathbb{C}}
\newcommand{\J}{\mathbf{J}}
\newcommand{\K}{\mathbf{K}}
\newcommand{\N}{\mathbb{N}}
\newcommand{\Z}{\mathbb{Z}}
\newcommand{\R}{\mathbb{R}}
\newcommand{\Sb}{\mathbb{S}}

\newcommand{\F}{\mathbf{F}}
\newcommand{\E}{\mathbf{E}}

\newcommand{\Omegabf}{\mathbf{\Omega}}

\renewcommand{\H}{\mathbf{ H}}

\title{Asymptotic number of scattering resonances for generic Schr\"odinger operators}

\author{Tien-Cuong Dinh and Duc-Viet Vu}

\begin{document}

\maketitle

\begin{abstract}
Let $-\Delta+V$ be  the Schr\"odinger operator acting on $L^2(\R^d,\C)$ with $d\geq 3$ odd. Here $V$ is a bounded real or complex function vanishing outside the closed ball of center 0 and of radius $a$. 
Let  $n_V(r)$ denote  the number of resonances of  $-\Delta+V$ with modulus $\leq r$.
We show that for generic potentials $V$ 
$$n_V(r)=c_d a^dr^d+ O(r^{d-{3\over 16}+\epsilon}) \quad \mbox{as } r\to\infty$$
for any $\epsilon>0$, where $c_d$ is a dimensional constant.
\end{abstract}

\noindent
{\bf Classification AMS 2010}: 35P25, 47A40. 


\section{Introduction} \label{intro}

Let $\Delta$ denote the Laplacian on $\R^d$ with $d$ odd. Let $V$ be a bounded complex-valued function with support in the closed ball $\overline \B_a$ of  center $0$ and of radius $a$ in $\R^d$. The purpose of this work is to study the asymptotic number of resonances associated to the Schr\"odinger operator $-\Delta+V$ acting on $L^2(\R^d,\C)$. The problem has a long history and was intensively investigated during the last three decades. We refer to \cite{Christiansen1, Christiansen3, Sjostrand, Stefanov,Vodev2, Zworski4, Zworski5,Zworski6} and to the references therein for an introduction to the subject.

Recall that  for $\lambda\in \C$ large enough with $\Im(\lambda)>0$, the operator $R_V(\lambda):=(-\Delta+V-\lambda^2)^{-1}$ on $L^2(\R^d,\C)$ is well-defined and is bounded. It depends holomorphically on the parameter $\lambda$. 
If $\chi$ is a smooth function with compact support such that $\chi V=V$, one can extend $\chi R_V(\lambda)\chi$ to a family of operators which depends meromorphically on $\lambda\in\C$. The poles of this family, which 
are called {\it the resonances} of the operator $-\Delta+V$,
and their multiplicities do not depend on the choice of $\chi$.  Denote by $n_V(r)$ the number of resonances of modulus $\leq r$ counted with multiplicity.

In dimension $d=1$, Zworski obtained in \cite{Zworski1} that 
\begin{equation} \label{e:dim_1}
n_V(r)={4\over \pi} ar  +o(r) \quad \mbox{as}\quad r\to\infty,
\end{equation}
where $2a$ is the diameter of the support of $V$, see also \cite{Froese, Regge, Simon,Zworski5}. From now on, we assume that $d\geq 3$.

The upper bound for the number of resonances is now well-understood. Define
$$N_V(r):=\int_0^{r} {n_V(t)-n_{V}(0)\over t} dt.$$
Generalizing some results by Melrose \cite{Melrose1, Melrose2, Melrose3}, Zworski obtained in \cite{Zworski3} the following estimate
$$dN_V(r)\leq c_da^dr^d+o(r^d) \quad \mbox{as}\quad r\to\infty,$$
where the sharp constant $c_d$ was identified by Stefanov in \cite{Stefanov}, see Section \ref{s:function} for the definition of $c_d$ and \cite{SZ, Vodev1} for more general results. Lower bound for the number of resonances is only known in special cases that we will discuss below.

For $0<\delta \leq 1$, define
$$\mathfrak{M}^\delta_a:=\Big\{V\in L^\infty(\B_a,\C):\  n_V(r)-c_da^dr^d=O(r^{d-\delta+\epsilon}) \mbox{ as } r\to\infty \mbox{ for every } \epsilon>0\Big\}.$$
This is a subset of the following family introduced by Christiansen in \cite{Christiansen4}
$$\mathfrak{M}_a:=\Big\{V\in L^\infty(\B_a,\C): \   n_V(r)-c_da^dr^d=o(r^d) \mbox{ as } r\to\infty\Big\}.$$
Our first main result is the following theorem. 
 
 \begin{theorem} \label{t:main_1}
 Let $V$ be a radial real-valued function of class $\Cc^2$ on $\overline \B_a$. Write $V(x)=V(\|x\|)$ and assume  that $V(a)\not=0$. Then $V$ belongs to $\mathfrak{M}_a^{3/4}$. 
 \end{theorem}

This result generalizes a theorem of Zworski in \cite{Zworski2} which says that $V$ belongs to $\mathfrak{M}_a$, see also Stefanov \cite[Th.3]{Stefanov}.
The proof will be given in Section \ref{s:radial}. It follows Zworski's approach and is based on some refinements of his arguments.

Consider now a connected open set $\Omega$ in $\C^p$ and a uniformly bounded family $V_\vartheta$ of potentials in $L^\infty(\B_a,\C)$ depending holomorphically on the parameter $\vartheta\in\Omega$. Our second main result is the following theorem.

\begin{theorem} \label{t:main_2}
Let $V_\vartheta$ be a holomorphic family of potentials as above. Suppose there are $\vartheta_0\in\Omega$ and $0<\delta\leq 1$ such that $V_{\vartheta_0}$ belongs to $\mathfrak{M}^\delta_a$. Then  there is a pluripolar set $E\subset \Omega$ such that $V_\vartheta\in \mathfrak{M}^{\delta/4}_a$ for all $\vartheta\in \Omega\setminus E$.
\end{theorem}

Note that pluripolar sets in $\Omega$ are of Hausdorff dimension at most equal to $2p-2$ and their intersections with $\R^p$ have zero $p$-dimensional volume, see e.g. \cite{DNS, LelongGruman, Ransford} and also Section \ref{s:holo} for the definition. Therefore, in the last theorem, most of potentials $V_\vartheta$ belong to $ \mathfrak{M}^{\delta/4}_a$. If $V$ is a potential as in Theorem \ref{t:main_1} and if $V'$ is an arbitrary potential in $L^\infty(\B_a,\C)$, then for almost every $\vartheta\in\C$ and almost every  $\vartheta\in\R$ the potential $\vartheta V+(1-\vartheta )V'$ belongs to $\mathfrak{M}_a^{3/16}$.  For such a potential, the number of resonances is asymptotically $c_da^dr^d$. Therefore, this property holds for generic potentials in $L^\infty(\B_a,\C)$ or in $L^\infty(\B_a,\R)$ if we only consider real potentials.

A version of Theorem \ref{t:main_2} has been obtained by Christiansen in \cite{Christiansen4}, where assuming $V_{\vartheta_0}\in \mathfrak{M}_a$, she proved that the counting function $N_{V_\vartheta}(r)$ satisfies
$$\limsup_{r\to\infty} {dN_{V_\vartheta}(r)\over r^d}=c_da^d$$ 
for $\vartheta$ outside a pluripolar set, see also \cite{Christiansen1,CH1}. In particular, this property holds for generic potentials $V$ in $L^\infty(\B_a,\C)$ or in $L^\infty(\B_a,\R)$. 

The proof of Theorem \ref{t:main_2} will be given in Section \ref{s:holo}. It partially follows  Christiansen's approach. We also prove and use there
some property of plurisubharmonic functions (see Lemma \ref{l:psh} below) and an upper bound for $N_V(r)$ which generalizes the above estimate (\ref{e:dim_1}) by Zworski and Stefanov, see Theorem \ref{t:upper_bound} in Section \ref{s:general_operator} below.

Note that Christiansen constructed in \cite{Christiansen2} examples of complex Schr\"odinger operators without resonances. This shows that the exceptional set $E$ in Theorem \ref{t:main_2} is not always empty. 
In comparison with similar results from complex dynamics, it is reasonable to believe that $E$ is always a finite or countable union of analytic subsets of $\Omega$, see e.g. \cite{DS}.

S\'a Barreto and Zworski showed in \cite{SaZ}  that any Schr\"odinger operator with compactly supported real potential admits an infinite number of resonances, see also \cite{Christiansen0, Sa}. The sharp asymptotic behavior for the number of resonances in this case is still unknown.    

\medskip
\noindent
{\bf Notation and convention.} Denote by $\B_a$ the open ball of center 0 and of radius $a$ in $\R^d$ and $\D(z,r)$ the disc of center $z$ and of radius $r$ in $\C$. Let $\Sb^{d-1}$ denote the unit sphere in $\R^d$.
Define $\B:=\B_1$, $\D(r):=\D(0,r)$, $\D:=\D(1)$, $\N^*:=\N\setminus\{0\}$, $\R_\pm:=\{t\in\R:\ \pm t\geq 0\}$ and
$\C_\pm:=\{z\in\C:\ \pm\Im z>0\}$.
The functions 
$\rho$, $\zeta$, $\Ai$, $J_\nu$, the constant $c_d$, the sets $\Omegabf, \K_+$ and the space $H_l$ are introduced in Section \ref{s:function}; the sets  $\Omegabf^\nu_c, \Omegabf^\nu_c(k),\J_c^\nu$  in Section \ref{s:radial}. 
Define $\arg z:=\theta$ and $\log z:=\log r+ i \theta$ for $z=re^{i\theta}$ with $r>0$ and $\theta\in(-\pi,\pi]$.
All the constants we will use depend only on $a,d,\|V\|_\infty$ and can be changed from line to line. 
The notation $\lesssim$ and $\gtrsim$ means inequalities up to a multiplicative constant.
An expression likes $f(z)\sim g(z)$ as $z\to\infty$ means $f(z)/g(z)\to 1$ when $|z|\to\infty$.
An expression likes 
$$f(z)\sim g(z)\sum_{n=0}^\infty {a_n\over z^n}$$
means 
$${f(z)\over g(z)} = \sum_{n=0}^{N-1} {a_n\over z^n} +O(z^{-N})\quad \mbox{as} \quad |z|\to\infty$$
for each $N\geq 0$, see \cite[p.16]{Olver2}. 

\medskip
\noindent
{\bf Acknowledgement.} We would like to thank Viet-Anh Nguyen, St\'ephane Nonnenmacher and Nessim Sibony for their remarks and for fruitful discussions on this subject.

\section{Some properties of Bessel functions} \label{s:function}

In this section, we give some properties of Bessel functions and of other auxiliary functions that will be used later in the proofs of the main theorems. We refer to Olver \cite{Olver1,Olver2} for details.

Let $\rho$ be the continuous function on $\overline\C_+\setminus\{0\}$  defined by 
\begin{equation} \label{e:rho_def}
\rho(z):=\log \frac{1+\sqrt{1-z^2}}{z}-\sqrt{1-z^2}
\end{equation}
which extends the real-valued function in $z\in(0,1)$ given by the same formula.
Let $\Omegabf$ be the following union of a half-plane and a half-strip
$${\mathbf\Omega}:= \big\{z\in\C:\ \Re z<0\big\}\cup \big\{z\in\C:\ -\pi<\Im z <0, \Re z\geq 0\big\}.$$ 
Then, the function $\rho$ defines a bijection between $\overline\C_+\setminus\{0\}$ and $\overline{\mathbf\Omega}$.
Moreover, it is holomorphic on $\C_+$ and sends  the intervals
$$[1,\infty), (0,1], [-1,0), (-\infty,-1]$$ 
respectively and bijectively to 
$$i\R_+, \R_+,  \big\{z\in\C:\ \Im z=-\pi, \Re z\geq 0\big\} \quad \mbox{and} \quad i(-\infty,-\pi].$$
A direct computation gives
\begin{equation} \label{e:rho}
{\partial\rho\over \partial z}=-{\sqrt{1-z^2}\over z} \quad \mbox{and} \quad |\rho|\sim\const |1-z|^{3/2} \quad \mbox{as} \quad z\to 1.
\end{equation}

As in \cite{Olver1,Olver2}, one can find an injective continuous function $\zeta:\overline\C_+\setminus\{0\}\to \overline\C_-$ which sends bijectively $(0,1]$ to $\R_+$ and satisfies
\begin{equation} \label{e:zeta}
{2\over 3}\zeta^{3/2}(z)=\rho(z).
\end{equation}
The function $\zeta$ is holomorphic on $\C_+$.

Consider the convex domain 
$$\K_+:=\big\{z\in \C_+:\ \Re\rho(z)>0\big\}.$$ 
Its boundary is the union of the interval $[-1,1]$ and 
the curve $\rho^{-1}([-i\pi,0])$ joining the two points $-1$ and $1$. This is the upper half of the domain $\K$ considered in \cite{Olver1,Olver2}, \cite[p.126]{Stefanov} and \cite[p.377]{Zworski2}. Note that $\K_+$ contains the half-disc $\overline\D({1\over 2})\cap \C_+$.

Recall that the dimensional constant $c_d$ used in Introduction was defined in \cite{Stefanov,Zworski2, Zworski3}. It is equal to
\begin{eqnarray} \label{e:c_d}
c_d & = & {2d\over \pi(d-2)!}\int_{z=x+iy\in\C_+} {\max(-\Re\rho(z),0)\over |z|^{d+2}}dxdy \nonumber\\
& = & {2d\over \pi(d-2)!}\int_{z=x+iy\in \C_+\setminus\K_+} -{\Re\rho(z)\over |z|^{d+2}}dxdy \nonumber\\
& = & {2\volume(\B)^2\over (2\pi)^d} +{2\over \pi d(d-2)!} \int_{\partial\K_+}{|1-z^2|^{1/2}\over |z|^{d+1}}|dz|.
\end{eqnarray}

We will need some basic properties of the Airy function $\Ai(\cdot)$, of its derivative $\Ai'(\cdot)$ and of the Bessel function $J_\nu(\cdot)$ with a large positive parameter $\nu$. The functions $\Ai(\cdot)$ and $\Ai'(\cdot)$ are entire. The function $J_\nu(\cdot)$ is holomorphic on $\C\setminus\R_-$. 
For $w\in\C\setminus\R_-$, define $\xi:={2\over 3} w^{3/2}$, where we use the principal branch for the function $w\mapsto w^{3/2}$. 
 There are real numbers $u_s$ and $v_s$ such that 
\begin{equation} \label{e:Ai_1}
\Ai(w) \sim \frac{e^{-\xi}}{2\pi ^{1/2}w^{1/4} } \sum_{s=0}^{\infty}\frac{u_s}{(-\xi)^s}
\quad \mbox{and} \quad \Ai'(w) \sim \frac{w^{1/4}e^{-\xi}}{2\pi ^{1/2}}\sum_{s=0}^{\infty}\frac{v_s}{(-\xi)^s} 
\end{equation}
as $|w|\to\infty$ in $|\arg w|\leq \pi-\delta$ for every fixed constant $\delta>0$. 

For the values of $\Ai(\cdot)$ and $\Ai'(\cdot)$ on $\C\setminus \R_+$, we need other formulas. With the above notation, there are real numbers $a_s,b_s,a'_s,b'_s$  such that 
\begin{equation} \label{e:Ai_2}
\Ai(-w) \sim \frac{1}{\pi^{1/2}w^{1/4}}\bigg[\cos(\xi-\frac{\pi}{4})\big(1+\sum_{s=1}^{\infty}\frac{a_s}{\xi^{2s}}\big)+\sin(\xi-\frac{\pi}{4})\sum_{s=0}^{\infty}\frac{b_s}{\xi^{2s+1}}\bigg]
\end{equation}
and 
\begin{equation} \label{e:Ai_3}
\Ai'(-w) \sim \frac{w^{1/4}}{\pi^{1/2}}\bigg[\sin(\xi-\frac{\pi}{4})\big(-1+\sum_{s=1}^{\infty}\frac{a'_s}{\xi^{2s}}\big)+\cos(\xi-\frac{\pi}{4})\sum_{s=0}^{\infty}\frac{b'_s}{\xi^{2s+1}}\bigg]
\end{equation}
as $|w|\to\infty$ with
$|\arg w|\leq {2\pi\over 3}-\delta$ for every fixed constant $\delta>0$.

For the Bessel function $J_\nu(\cdot)$, when $\nu\to\infty$, the following relation holds 
uniformly in $0\leq \arg z\leq \pi-\delta$ with any fixed constant $\delta>0$ 
\begin{equation}\label{e:Bessel_1}
J_\nu(\nu z) \sim \big(\frac{4\zeta}{1-z^2}\big)^{1/4}\bigg[ \frac{\Ai(\nu^{2/3}\zeta)}{\nu^{1/3}}\big(1+\sum_{s=1}^{\infty}\frac{A_s(\zeta)}{\nu^{2s}}\big)+\frac{\Ai'(\nu^{2/3}\zeta)}{\nu^{5/3}}\sum_{s=0}^{\infty}\frac{B_s(\zeta)}{\nu^{2s}}\bigg],
\end{equation}
where $A_s$ and $B_s$ are holomorphic functions in $\zeta$, see Olver \cite[(4.24)]{Olver1}.  Note that a similar property holds for $z\in \C\setminus (-\infty ,1]$ and $0\leq |\arg z|\leq \pi-\delta$ by Schwarz reflection principle. 

\medskip

We will need the following estimates.

\begin{lemma} \label{l:Bessel_1}
Let $M>0$ be a fixed constant large enough. Then, there is a constant $A>0$ such that for $\nu$ large enough and for $z\not=0$ with $\Re(z)\geq 0$, $\Im(z)\geq 0$ we have
$$|J_\nu(\nu z)|\le A\max(1,-\log|z|)e^{-\nu \Re \rho} \quad \mbox{when} \quad \nu^{2/3}|1-z|\geq M$$
and 
$$|J_\nu(\nu z)|\leq A \quad \mbox{when} \quad \nu^{2/3}|1-z|\leq M.$$
\end{lemma}
\proof
Assume that  $\nu^{2/3}|1-z|\leq M$. Then $z$ is close to 1. We deduce from (\ref{e:rho}) and (\ref{e:zeta}) that 
$|\zeta|\sim \const |1-z|$. So
$|\zeta|$, $|\nu^{2/3}\zeta|$ and the first factor in the right-hand side of (\ref{e:Bessel_1}) are bounded. Therefore, we deduce from (\ref{e:Bessel_1}) that $|J_\nu(\nu z)|$ is bounded.

Assume now that $\nu^{2/3}|1-z|\geq M$. Then $|\nu^{2/3}\zeta|$  and $|\nu\rho|$ are bounded below by a large positive constant. This allows us to use the identities (\ref{e:Ai_1}), (\ref{e:Ai_2}) and (\ref{e:Ai_3}).
We distinguish two cases. Consider first the case where  $-\pi\leq \arg(\zeta)\leq -{\pi\over 2}$. In this case, we have $\Re \rho\leq 0$. Then, we can apply the relations
 (\ref{e:Bessel_1}), (\ref{e:Ai_2}) and (\ref{e:Ai_3}) 
 to $w:=-\nu^{2/3}\zeta$ and $\xi:=-i\nu\rho$.
 We have
\begin{eqnarray*}
|J_\nu(\nu z)| & \lesssim & {|w|^{1/4}|\zeta|^{1/4}\over \nu^{1/3} |1-z^2|^{1/4}}\Big(\big|\sin\big(-i\nu\rho-{\pi\over 4}\big)\big|+\big|\cos\big(-i\nu\rho-{\pi\over 4}\big)\big|\Big)\\
& \lesssim & {|\zeta|^{1/2}\over |1-z^2|^{1/4}} e^{-\nu\Re\rho}.
\end{eqnarray*}
We obtain the result using that $|\zeta|\lesssim -\log|z|$ as $z\to 0$, $|\zeta|\lesssim|1-z|$ as $z\to 1$ and $|\zeta|\lesssim |z|^{2/3}$ as $z\to\infty$.

It remains to treat the case where  $-{\pi\over 2} \leq \arg(\zeta)\leq 0$. In this case, we do not need to know the sign of $\Re\rho$. We can apply the relations
 (\ref{e:Bessel_1}) and (\ref{e:Ai_1}) 
 to $w:=\nu^{2/3}\zeta$ and $\xi:=\nu\rho$. Similar estimates as above give the result.
\endproof

Recall that the zeros of the function $J_\nu(\nu z)$, except 0,  are real, simple and larger than 1, see \cite[(7.4)]{Olver1}. So the corresponding values of $\rho$ belong to $i\R_+$. 
Fix an integer $k_0$ large enough. We say that a solution of $J_\nu(\nu z)=0$ is of first type if the corresponding value of $\rho$ satisfies $|\nu\rho|<k_0\pi$ and of second type otherwise.
Let $\widetilde z_{\nu,k_0},\widetilde z_{\nu,k_0+1},\ldots$ be the solutions of second type of $J_\nu(\nu z)=0$ written in increasing order. Define  $\widetilde\rho_{\nu,k}:=\rho(\widetilde z_{\nu,k})$. We will need later the following lemma.

\begin{lemma} \label{l:zero_Bessel_bis}
For $\nu$ large enough the number of solutions of first type of $J_\nu(\nu z)=0$ is bounded by a constant independent of $\nu$. Moreover, 
there is a constant $\epsilon_0>0$ such that for $\nu$ large enough and for $k_0\leq k\leq \epsilon_0\nu^4$, we have 
$$\Big|\widetilde\rho_{\nu,k}-\big({3\pi i\over 4\nu} +{k\pi i\over \nu}\big)\Big|\leq {1\over\nu}\cdot$$
\end{lemma}
\proof
Assume that $\nu$ is large enough. Consider the solutions $z$ of first type. 
As above, we can apply (\ref{e:Bessel_1}) and (\ref{e:Ai_2}), (\ref{e:Ai_3}) to $w:=-\nu^{2/3}\zeta$ and $\xi:=-i\nu\rho$. We can see using Rouch\'e theorem that $w$ is almost equal to a solution of $\Ai(-w)=0$ in a bounded interval.
So the number of solutions of first type is finite.

We prove now the second assertion in the lemma.
Recall that the function $\rho$ sends bijectively $[1,\infty)$ to $i\R_+$. So the $\widetilde\rho_{\nu,k}$ are in $i\R_+$ and the sequence $|\widetilde\rho_{\nu,k}|$ is increasing. We will only consider the zeros of $J_\nu(\nu z)$ such that 
$k_0\pi\nu^{-1}\leq |\rho|< \epsilon \nu^3$ for some fixed small constant $\epsilon>0$. For such a zero, we have $|\nu^{2/3}\zeta|< 2\epsilon \nu^{8/3}$. 

We apply again (\ref{e:Bessel_1}) and (\ref{e:Ai_2}), (\ref{e:Ai_3}) to $w:=-\nu^{2/3}\zeta$ and $\xi:=-i\nu\rho$. Using that $|w^{1/4}|\ll \nu^{2/3}$, we see that
$\xi$ is a positive number large enough (because $k_0$ is a large constant) satisfying
an equation of the form
$$\cos \big(\xi-{\pi \over 4}\big) = \gamma_\nu(\xi),$$
where $\gamma_\nu(\xi)$ is a holomorphic function on the domain
$$\big\{\xi\in\C:\ |\xi|<\epsilon\nu^4,\Re\xi>0, |\Im\xi|<1\big\}$$
such that $|\gamma_\nu|$ is bounded by a very small constant independent of $\nu$. 
We use here the property that  $\cos(\xi-{\pi\over 4})$ and $\sin(\xi-{\pi\over 4})$ are bounded on the considered domain.  

Choose a constant $\epsilon_0\ll \epsilon$. We can now apply Rouch\'e's theorem and deduce that the first $\epsilon_0\nu^4-k_0+1$ zeros of second type of $J_\nu(\nu z)$ satisfy the lemma. 
\endproof

Let $H_l$ denote the vector space of harmonic 
homogeneous polynomials of degree $l$ on $\R^d$. These polynomials are used to describe the eigenfunctions  of the Dirichlet Laplacian on the unit ball $\B$ of $\R^d$. We recall some basic properties that will be used later. For the following result, see
Corollary 2.53  in \cite{Folland}.

\begin{lemma} \label{l:H_l}
We have
$$\dim H_l= (2l+d-2)\frac{(l+d-3)!}{l!(d-2)!}=\frac{2l^{d-2}}{(d-2)!}+ O(l^{d-3}) \quad \mbox{as} \quad l\to\infty.$$
\end{lemma}

Define $\nu:=l+{d\over 2}-1$.
Let $m^+_l(r)$ denote the number of positive zeros of  $J_\nu(\cdot)$ which are smaller or equal to $r$. It is almost equal to the maximal integer $k$ such that $\widetilde z_{\nu,k}\leq {r\over\nu}$. The following result is a consequence of the classical Weyl law.

\begin{proposition} \label{p:Weyl} 
Assume that $r$ is large enough. Then 
\begin{enumerate}
\item[{\rm (a)}] $m^+_l(r)=0$ when $l\geq r$;
\item[{\rm (b)}] $\widetilde z_{\nu,k} \not\in \overline\D({r\over\nu})$ if $k\geq {1\over 2}cr$ for some fixed constant $c>0$ large enough; in particular, we have $m^+_l(r)\leq cr$;
\item[{\rm (c)}] We have
$$\sum_{l\geq 0} m^+_l(r)\dim H_l ={\volume(\B)^2\over (2\pi)^d} r^d +O(r^{d-1}) \quad \mbox{as} \quad r\to\infty.$$
\end{enumerate}
\end{proposition}
\proof
(a) If $l\geq r$ then $\nu>r$. In this case, if $z>0$ is a solution of 
$J_\nu(\nu z)=0$, by classical properties of Bessel functions, we have $z\geq 1$. Therefore, 
$\nu z>r$ and hence, $m^+_l(r)=0$. 

\medskip\noindent
(b) Assume that $l\leq r$ and $k\geq {1\over 2}cr$ for $c>0$ large enough. Then, by Lemma \ref{l:zero_Bessel_bis}, $\widetilde\rho_{\nu,k}$ is large and therefore $\widetilde z_{\nu,k}$ is larger than 2. For $z$ in $[2,\infty)$, we have
$$\rho= \Big(\sqrt{z^2-1}+\arccos{1\over z}\Big) i.$$
We deduce that 
$\nu\widetilde z_{\nu,k} \gtrsim \nu|\widetilde\rho_{\nu,k}|\gtrsim k$ and hence $\widetilde z_{\nu,k}>{r\over \nu}$.

\medskip\noindent
(c)  Recall that the eigenvalues of the Dirichlet Laplacian on $\B$ are precisely $(\nu \widetilde z_{\nu,k})^2$ 
with multiplicity $\dim H_l$, see Theorem 2.66 in \cite{Folland} for details. So the infinite sum in the proposition is the number of eigenvalues $\leq r^2$ of the Dirichlet Laplacian on $\B$ counted with multiplicities. By Weyl law \cite[Th. 29.3.3]{Hormander}, this number is equal to 
$${\volume(\B)^2\over (2\pi)^d} r^d +O(r^{d-1}) \quad \mbox{as}\quad   r\to\infty.$$ 
This completes the proof of the proposition.
\endproof

\section{Upper bound for the number of resonances} \label{s:general_operator}

In this section, we obtain an upper bound for the number of resonances which improves a result due to Zworski-Stefanov \cite{Stefanov, Zworski1}. 
Consider a general Schr\"odinger operator $-\Delta+V$ with a bounded complex potential $V$ vanishing outside the ball $\B_a$. 
Here is the main result in this section which is a consequence of Proposition \ref{p:Christiansen} and Theorem \ref{t:Stefanov} below. 

\begin{theorem} \label{t:upper_bound}
With the notation as in Introduction, there is a constant $A>0$ depending only on $d,a$ and $\|V\|_\infty$ such that
$$N_V(r)\leq {c_da^d r^d \over d}+Ar^{d-1}\log r \quad \mbox{as } r\to \infty.$$
\end{theorem}
 
 We first recall some basic notions and results, see \cite{Stefanov, Zworski1} for details.  Let $R_1,R_2,R_3$ be real numbers such that $a<R_1<R_2<R_3$.  Choose also a smooth cut-off function $\chi_1$ (resp. $\chi_2$) vanishing outside $\B_{R_2}$ (resp. $\B_{R_3}$) and equal to 1 on $\B_{R_1}$ (resp. $\B_{R_2}$). These numbers and functions will be specified later. 

Define two families of operators $\E_\pm(\lambda):L^2_\comp(\R^d)\to L^2(\Sb^{d-1})$ with $\lambda\in\C$ by 
$$\E_\pm(\lambda)(f)(w):=\int e^{\pm i\lambda\omega\cdot x} f(x)dx \quad \mbox{for } f \in L^2_\comp(\R^d) \mbox{ and } \omega\in\Sb^{d-1}.$$
Denote by $\E_+^*(\lambda)$ the transpose operator of $\E_+(\lambda)$ with the same Schwartz kernel. 
The scattering matrix associated to $-\Delta+V$ is the operator  $S_V(\lambda):L^2(\Sb^{d-1})\to L^2(\Sb^{d-1})$ given by 
$$S_V(\lambda):=I-i(2\pi)^{-d}2^{{1-d\over 2}} \lambda^{d-2} \E_-(\lambda)[\Delta,\chi_2]R_V(\lambda)[\Delta,\chi_1] \E_+^*(\lambda),$$
where $I$ denotes the identity operator.

{\it The scattering determinant} is defined by
$$s_V(\lambda):=\det S_V(\lambda).$$
It satisfies $s_V(\lambda)s_V(-\lambda)=1$. The poles of $s_V$ are called {\it the scattering poles}. They are, with a finite number of exceptions,  the resonances of $-\Delta+V$ with the same multiplicities. In what follows, we will tend to abuse notation and identify $n_V(r)$, $N_V(r)$ with the similar counting functions for the zeros of $s_V(\lambda)$ on $\C_+$. This does not affect our estimates. 

\medskip

The following result was obtained by Christiansen in \cite[(3.2)]{Christiansen4}, see also Stefanov \cite[Prop. 2]{Stefanov}. 

\begin{proposition}\label{p:Christiansen}
We have  for $r$ large enough
$$ \Big|N_V(r)- \frac{1}{2\pi} \int_0^{2\pi} \log |s_{V}(r e^{i\theta})| d\theta\Big|
\leq Ar^{d-1},$$
where $A>0$ is a constant depending only on $d,a$ and $\|V\|_\infty$.
\end{proposition}

We have the following refinement of \cite[Th. 5]{Stefanov} where the function $h_d$ is defined by 

$$h_d(\theta):={4\over (d-2)!} \int_0^\infty {\max(-\Re \rho(te^{i\theta}),0)\over t^{d+1}} dt 
\quad \mbox{for} \quad \theta\in[0,\pi].$$
This function is continuous, positive and satisfies $h_d(0)=h_d(\pi)=0$.

\begin{theorem}\label{t:Stefanov}
With the above notation, there is a constant $A>0$  depending only on  $d,a$ and $\|V\|_\infty$ such that 
$$\log |s_V(re^{i\theta})| \le h_d(\theta)a^d r^d+ Ar^{d-1}\log r$$
for all $r$ large enough and  $\theta\in [0,\pi]$.
\end{theorem}

In the rest of the section, we give the proof of Theorem \ref{t:Stefanov}. 
Observe that we can suppose $\theta\leq \pi/2$ since otherwise we can reduce the problem to the first case by replacing $V$ with $\overline V$ and $\lambda$ with $-\overline \lambda$.
By rescaling, we can also assume that $a=1$. Choose $R_j:=1+{j\over r}$ for $j=1,2,3$. Choose $\chi_1,\chi_2$ as above such that $\|\chi_j\|_{\Cc^2}\leq cr^2$ for some constant $c>0$ independent of $r$. 
Define for $l\geq 1$, $\lambda\in\C$ with $\Re\lambda\geq 0$, $\Im\lambda\geq 0$ and $0<s\leq s'$
$$I_l(\lambda, s,s'):= \int_{s}^{s'} |\lambda|^{2-d} |J_{l+d/2-1}(\lambda t)|^2 t dt$$
and 
$$\mu_l^\star(\lambda):= (2\pi)^dI_l(\lambda,R_1,R_2)^{1/2} I_l(\lambda,R_2,R_3)^{1/2} \quad \mbox{for} \quad \lambda \in\C_+.$$

The following lemma refines an estimate obtained by Stefanov.

\begin{lemma} \label{l:Stefanov}
There is a constant $A>0$ such that 
\begin{align}\label{e:Stefanov}
\log |s_V(re^{i\theta})|\le \sum_{l=1}^{\infty}(\dim H_l)\log \big(1+ Ar^{d+4}\mu^\star_l(re^{i\theta}) \big)
\end{align}
 for $r$ large enough and $0 \leq \theta\leq {\pi\over 2}\cdot$
 \end{lemma}
\proof
If $Q:H\to H$ is a bounded linear operator on a Hilbert space such that the spectrum of $(Q^*Q)^{1/2}$ is discrete, denote by $\mu_1(Q),\mu_2(Q),\ldots$ the singular values of $Q$, i.e. the eigenvalues of $(Q^*Q)^{1/2}$, written in decreasing order and repeated according to their multiplicities. Define
$$\H(\lambda):=-i(2\pi)^{-d}2^{1-d\over 2}\lambda^{d-2}[\Delta,\chi_2]R_V(\lambda)[\Delta,\chi_1]$$
and
$$\F(\lambda):=\ind_{R_1< \|x\|<R_2} \E^*_{+}(\lambda) \E_{-}(\lambda) \ind_{R_2< \|x\|<R_3}.$$

In the proof of 
Theorem 5 in \cite[p.128]{Stefanov}, Stefanov obtained that 
\begin{eqnarray*}
\log |s_V(re^{i\theta})| & \le &  \sum_{l=1}^{\infty}\log \big(1+  \mu_l (\H(re^{i\theta})\F(re^{i\theta}))\big)\\
& \leq &  \sum_{l=1}^{\infty}\log \big(1+  \|\H(re^{i\theta})\|_{L^2\to L^2}\mu_l (\F(re^{i\theta}))\big).
\end{eqnarray*}
The last inequality is a consequence of the general inequality $\mu_l(AB)\leq \|A\|\mu_l(B)$. 

Stefanov also proved that, up to a permutation of elements,  the sequence $\mu_l(\F(re^{i\theta}))$ is constituted by the $\mu_l^\star(re^{i\theta})$'s
where each number $\mu_l^\star(re^{i\theta})$ is repeated  $(\dim H_l)$ times. 
Since we only consider sums of positive numbers, this permutation does not affect our computation.
So we only have to check that 
$$\|\H(re^{i\theta})\|_{L^2\to L^2}\leq Ar^{d+4}$$
for $r$ large enough and for a large fixed constant $A>0$.

Choose a smooth function $\rho\leq 1$ with compact support which is equal to 1 on $\B_{R_3}$ and with bounded $\Cc^1$-norm. 
Since the operators $[\Delta,\chi_i]$ are of order 1, using the above estimates on $\chi_i$, we only need to check that 
$$\|\rho R_V(re^{i\theta})\rho\|_{H^{-1}\to H^1}=O(r^2).$$
But,  this estimate is a consequence of the classical theory of elliptic operators,  see  Zworski \cite{Zworski6} for details. The lemma follows. 
\endproof

\noindent
{\bf Proof of Theorem \ref{t:Stefanov}.}  
Recall that we only have to consider the case where $0\leq\theta\leq {\pi\over 2}$ and we have to bound the right-hand side in (\ref{e:Stefanov}). Using (\ref{e:Ai_1}), (\ref{e:Ai_2}), (\ref{e:Ai_3}) and (\ref{e:Bessel_1}), we see that $\log \mu^\star_l(re^{i\theta})\lesssim r$. Therefore, we only have to consider $l$ larger than any fixed constant.

Let $M$ be the constant in Lemma \ref{l:Bessel_1}. Define for 
$\nu:=l+{d\over 2}-1$ 
$$N:=\Big\{l\in\N^*:\ \nu^{2/3} |1-{tr e^{i\theta}\over\nu}|<M \ \mbox{ for some } \  t\in (R_1,R_3)\Big\},\quad N':=\N^*\setminus N$$
and 
$$N'_1:=\Big\{l\in N':\ {r R_3\over \nu} \not\in \K_+\Big\}, \ 
N'_2:=\Big\{l\in N':\ {rR_3\over \nu}\leq {1\over 100}\Big\},\ N'_3:=N'\setminus (N'_1\cup N'_2).$$
Denote by $\Sigma,\Sigma',\Sigma_i'$ the sums as in the right-hand side of (\ref{e:Stefanov}) but only with $l$ running in $N,N'$ or $N'_i$ respectively. We will bound these sums separately. The theorem is a direct consequence of the estimates given in the 4 cases below. 

\medskip
\noindent
{\bf Case 1.} Assume that $l\in N$. 
Since $R_1,R_2,R_3$ are close enough to  each  other and $\nu$ is large,
we have 
 ${1\over 2}<\frac{tr}{\nu}<2$ for all $t\in (R_1,R_3)$. In particular, we have $\nu<2rR_3<4r$. Moreover, for all $t,t'\in(R_1,R_3)$
 $$\nu^{2/3} \Big|1-{tr e^{i\theta}\over\nu}\Big|-\nu^{2/3} \Big|1-{t'r e^{i\theta}\over\nu}\Big|\leq \nu^{2/3}{(R_3-R_1)r\over\nu}<1.$$
 It follows that 
 $$\nu^{2/3} \Big|1-{tr e^{i\theta}\over\nu}\Big|\leq M+1 \quad \mbox{for every}\quad t\in(R_1,R_3).$$
Applying the second assertion of Lemma \ref{l:Bessel_1} to ${tre^{i\theta}\over \nu}$ and to $M+1$ instead of $z$ and $M$ yields 
$$\Sigma \lesssim \sum_{\nu< 4r} (\dim H_l) \log r
\lesssim  \sum_{l< 4r} l^{d-2} \log r \lesssim r^{d-1} \log r.$$

\medskip
\noindent
{\bf Case 2.} Assume now that $l\in N'=N'_1\cup N'_2\cup N'_3$. Observe that the function $t\mapsto -\Re \rho(tre^{i\theta})$ is increasing since we have by (\ref{e:rho})
$$-{\partial \Re\rho(tz)\over\partial t}={\Re\sqrt{1-(tz)^2}\over t}>0.$$
Therefore, by the first assertion in Lemma \ref{l:Bessel_1}, we have for some constant  $A>0$
\begin{equation} \label{e:sigma_2}
\Sigma'_i\leq \sum_{l\in N'_i} (\dim H_l) \log \Big(1+ A(\log\nu)^2r^{d+4} e^{-2 \nu \Re \rho\big({r e^{i\theta}R_3 \over \nu}\big)}\Big).
\end{equation}
\noindent
{\bf Case 2a.} Assume that $l\in N'_1$. We have $\nu\lesssim r$ and $-\Re \rho\big({r e^{i\theta}R_3 \over \nu}\big)\geq 0$. Hence, by (\ref{e:sigma_2})
\begin{eqnarray*}
\Sigma_1' 
& \le &     \sum_{l\in N_1'} ( \dim H_l) \Big[\log \big(A(\log \nu)^2 r^{d+4}\big)-2\nu \Re\rho\big({r e^{i\theta} R_3 \over \nu}\big)\Big] \\
&\le  &  \sum_{l\in N_1' } -\frac{4\nu^{d-1}+O(\nu^{d-2})}{(d-2)!} \Re \rho\big({r e^{i\theta} R_3 \over \nu}\big)+ O(r^{d-1}\log r).
\end{eqnarray*}
Since the function $t\mapsto -\Re \rho(te^{i\theta})$ is increasing, we deduce from the last estimates that 
\begin{eqnarray*}
\Sigma_1' & \leq &   
\int_{{r e^{i\theta} R_3 \over \nu} \not \in \K_+}  -\frac{4\nu^{d-1}+O(\nu^{d-2})}{(d-2)!} \Re \rho\big({r e^{i\theta} R_3 \over \nu}\big) d\nu + O(r^{d-1}\log r) \\
& \leq & {4(r R_3)^d \over (d-2)!} \int_{t e^{i\theta} \not \in \K_+} \frac{-\Re\rho(t e^{i\theta})}{t^{d+1}}\, dt + O(r^{d-1}\log r) \\
& \leq & {4r^d \over (d-2)!} \int_{t e^{i\theta} \not \in \K_+} \frac{-\Re\rho(t e^{i\theta})}{t^{d+1}}\, dt + O(r^{d-1}\log r).
\end{eqnarray*}

\medskip
\noindent
{\bf Case 2b.} Assume that $l\in N'_2$. Since $r$ is large, we have $l\geq 90r$. Observe that $\Re\rho(z)\ge -\log|z|-2$ when $|z|\leq {1\over 100}$. Hence, using that $\log(1+t)\leq t$ for $t\geq 0$, we obtain  from (\ref{e:sigma_2}) that
$$\Sigma_2' \lesssim   \sum_{l\geq 90r} l^{d-2} (\log l)^2 r^{d+4}e^{-2 l(\log l-\log r-3)} 
\lesssim     \sum_{l\geq 90r} l^{2d+3} e^{-l}. $$
It follows that $\Sigma_2'$ is bounded above.

\medskip
\noindent
{\bf Case 2c.} Assume that $l\in N'_3$. We have $l\leq \nu\leq 100rR_3<200r$. Since $\Re\rho({r e^{i\theta}R_3 \over \nu})$ is positive,
we obtain from (\ref{e:sigma_2}) that
$$\Sigma_3' \lesssim  \sum_{l\le 200r} l^{d-2} \log r  \ = \  O(r^{d-1}\log r). $$
This completes the proof of the theorem.      
\hfill $\square$

%%%%%%

\section{Schr\"odinger operators with radial potentials} \label{s:radial}

In this section, we give the proof of Theorem \ref{t:main_1}. We assume that the potential $V=V(\|x\|)$ satisfies the hypotheses of this theorem. 
By rescaling, we reduce the problem to the case $a=1$. 
 Define for $c>0$ 
$$\Omegabf_c^\nu:=\Big\{\rho\in \C: \ \Re \rho<-{\log(c\nu)\over 2\nu}\Big\}$$
and  
$$\J_c^\nu:=\rho^{-1}(\Omegabf^\nu_c)=\Big\{z\in\C_+:\ \Re\rho(z)<-{\log(c\nu)\over 2\nu}\Big\}.$$

Following Zworski \cite[p.400]{Zworski2}, the scattering poles are related to the zeros in $\J_c^\nu$ of a family of holomorphic functions
of the form (our notation is slightly different from Zworski's one)
\begin{equation} \label{e:pole}
g_\nu(z)={e^{-2\nu\rho(z)}\over \nu^2(1-z^2)}(1+\epsilon_\nu(z))-\sigma(1+\epsilon'_\nu(z)),
\end{equation}
where $\sigma$ is some complex number with $|\sigma|$ bounded below and above by positive constants and
$\epsilon_\nu(z)$, $\epsilon'_\nu(z)$ are continuous functions on $\overline \J_c^\nu$ which converge uniformly to 0 when $\nu\to\infty$. 
In comparison with Zworski's notation, for our convenience, we work with variable $z$ in $\C_+$ instead of $\C_-$. 

We need to compare $g_\nu$ with an auxiliary function  $h_\nu$ defined by
\begin{equation} \label{e:pole_almost}
h_\nu(z):={e^{-2\nu\rho(z)}\over \nu^2(1-z^2)}-\sigma.
\end{equation}
In what follows, we often consider $g_\nu$ and $h_\nu$ as functions on variable $\rho=\rho(z)$.

Let $f$ be the bi-holomorphic map from $\Omegabf=\rho(\C_+)$ to $\C\setminus (-\infty,1]$ defined by 
$$f(\rho(z)):=1-z^2.$$
We can extend it to a continuous map $f:\overline\Omegabf\to \C\setminus\{1\}$ which is no more bijective. A direct computations using (\ref{e:rho}) gives
\begin{equation} \label{e:f_log_f}
{\partial f\over \partial \rho}=-{2z^2\over \sqrt{1-z^2}} \quad  \mbox{and} \quad 
{\partial \log f\over \partial \rho}=-{2z^2\over (1-z^2)^{3/2}}\cdot
\end{equation}
We deduce for $z\in\overline\C_+\setminus\{0\}$ outside a neighbourhood of $-1$ (in particular, for $\Im\rho\geq -{3\pi\over 4}$) that  
\begin{equation} \label{e:log_f}
\Big|{\partial \log f\over \partial\rho}\Big|\lesssim 1+{1\over |\rho|} \quad \mbox{when}\quad \Re \rho<0.
\end{equation}

Define for $k\in\Z$
$$\Omegabf^\nu_c(k):=\Big\{\rho\in\Omegabf_c^\nu: \ \big|\Im \rho -\big(-{\arg\sigma\over 2\nu}  +{k\pi \over \nu}+{\pi\over 4\nu}\big)\big| < {\pi\over 2\nu}\Big\}.$$
These half-strips are disjoint and the union of their closures is equal to $\overline\Omegabf^\nu_c$.

\begin{lemma} \label{l:h_function}
Assume that $\nu$ is large enough and $k\geq -{\nu\over 2}-2$. Then, 
there is a constant $A>0$ independent of $\nu$ and $k$ such that
$|h_\nu(z)|\geq A$ for $\rho$ in the boundary of  $\Omegabf^\nu_c(k)$ and also for $\rho$ large enough in 
this domain.
\end{lemma}
\proof
For $\rho$ large enough in  $\Omegabf^\nu_c(k)$, $-\Re\rho$ is a large positive number and $|z|\sim -\Re\rho$. So  $|h_\nu(z)|$ is a big number. Consider now the case where $\rho$ belongs to the boundary of  $\Omegabf^\nu_c(k)$.  For $k\geq -{\nu\over 2}-2$ and for $\nu$ large enough, $\Im\rho$ is almost larger than $-{\pi\over 2}$ and hence $\Re z$ is almost positive. Therefore, $\arg f(\rho)$ belongs to the interval $(-\pi,\delta)$ for some small positive constant $\delta$ independent of $\nu,k$. 

Assume first that $\rho$ belongs to  the horizontal part of $b\Omegabf^\nu_c(k)$ which is the union of two half-lines given by
$$\Re\rho\leq -{\log(c\nu)\over 2\nu} \quad \mbox{and} \quad \Im \rho =-{\arg\sigma\over 2\nu}  +{k\pi \over \nu}+{\pi\over 4\nu} \pm {\pi\over 2\nu}\cdot$$
The above discussion on $\arg f(\rho)$ implies that $|\arg (\sigma^{-1}h_\nu(z)+1)|$ is bounded below by ${\pi\over 2}-\delta$. It follows that $|h_\nu(z)|$ is bounded below by a positive constant.

It remains to consider the case where $\rho$ belongs to the vertical part of $b\Omegabf^\nu_c(k)$. We have
$$\Re\rho= -{\log(c\nu)\over 2\nu}\cdot$$
Therefore, 
$$|h_\nu(z)+\sigma|={c\over \nu|1-z^2|}\cdot$$
It is enough to check that the last quantity is small. Since $\nu$ is large, this is clear when $z$ is outside a fixed neighbourhood of 1. Otherwise, we deduce from (\ref{e:rho}) that 
$$|1-z^2|\gtrsim |\rho|^{2/3}\gtrsim |\Re\rho|^{2/3}\gtrsim \nu^{-2/3}.$$
The result follows.
\endproof

\begin{lemma} \label{l:g_h}
Assume that $\nu$ is large enough and $k\geq -{\nu\over 2}-2$. Then $g_\nu$  and $h_\nu$, as functions on $\rho$, have the same number of zeros in $\Omegabf^\nu_c(k)$ counted with multiplicity.
\end{lemma}
\proof
We have for $\rho$ as in Lemma \ref{l:h_function}
$$|g_\nu(z)-h_\nu(z)|\leq 2\max\big(|\epsilon_\nu(z)|, |\epsilon'_\nu(z)|\big) \max\big(|\sigma|,|h_\nu(z)+\sigma|\big).$$
Lemma \ref{l:h_function} implies that the last factor is bounded by a constant times $|h_\nu(z)|$. Since $\epsilon_\nu$ and $\epsilon'_\nu$ are small, it is enough to apply Rouch\'e's theorem in order to obtain the result.
\endproof

\begin{lemma} \label{l:sol_unique}
Assume that $\nu$ is large enough. Then, for every $k\in\Z$ with $k\geq -{\nu\over 2}-2$, the function $g_\nu$ admits a unique zero in $\overline\Omegabf^\nu_c(k)$ that we denote by $\rho_{\nu,k}$. Moreover, this zero is simple and belongs to $\Omegabf_c^\nu(k)$.
\end{lemma}
\proof
First, we deduce from Lemma \ref{l:h_function} that $h_\nu$ and $g_\nu$ have no zero on the boundary of $\Omegabf_c^\nu(k)$. By Lemma \ref{l:g_h}, we only have to show that $h_\nu$ admits a unique zero in $\Omegabf_c^\nu(k)$ and this zero is simple.

The zeros of $h_\nu$ are exactly the solutions of  the following family of equations
\begin{equation} \label{e:pole_bis}
F_{\nu,k}(\rho)=0 \quad \mbox{with} \quad k\in\Z,
\end{equation}
where
$$F_{\nu,k}(\rho):=\rho -\Big(-{\log\sigma\over 2\nu} -{\log\nu\over \nu} -{\log f(\rho)\over 2\nu}+ {k\pi i\over\nu} \Big).$$
Therefore, we only have to prove that (\ref{e:pole_bis}) admits a unique solution in $\Omegabf_c^\nu$ which is simple and belongs to $\Omegabf_c^\nu(k)$.

Consider a solution $\rho\in \Omegabf^\nu_c$ of  (\ref{e:pole_bis}). 
By considering the equation $\Im F_{\nu,k}(\rho)=0$, we see that 
$\Im\rho\geq {(k-2)\pi\over\nu}$. So for $k\geq -{\nu\over 2}-2$ and for $\nu$ large enough, $\Im\rho$ is almost larger than $-{\pi\over 2}$. Therefore, $\arg f(\rho)$ belongs to the interval $(-\pi,\delta)$ for some small positive constant $\delta$. We deduce that 
$$\Big|\Im \rho -\big(-{\arg\sigma\over 2\nu}  +{k\pi \over \nu}+{\pi\over 4\nu}\big)\Big|
\leq \Big|-{\arg f(\rho)\over 2\nu}-{\pi\over 4\nu}\Big| <  {\pi\over 2\nu}\cdot$$
It follows that $\rho$ belongs to $\Omegabf^\nu_c(k)$. 

We now use the classical argument principle in order to count the number of zeros of $F_{\nu,k}$ in $\Omegabf^\nu_c(k)$. Observe that $\Im F_{\nu,k}(\rho)$ is bounded on $\overline\Omegabf^\nu_c(k)$ and $\Re F_{\nu,k}(\rho)\to-\infty$ when $|\rho|\to\infty$ and $\rho\in\overline\Omegabf^\nu_c(k)$. We will show in particular  that $\Re F_{\nu,k}(\rho)$ changes sign twice on $b\Omegabf^\nu_c(k)$. 

Consider first  the horizontal part of $b\Omegabf^\nu_c(k)$ which is the union of two half-lines given by
$$\Re\rho\leq -{\log(c\nu)\over 2\nu} \quad \mbox{and} \quad \Im \rho =-{\arg\sigma\over 2\nu}  +{k\pi \over \nu}+{\pi\over 4\nu} \pm {\pi\over 2\nu}\cdot$$
As above, we obtain that $\arg f(\rho)$ belongs to $(-\pi,\delta)$. We then deduce that
 $\Im F_{\nu,k}(\rho)$ is strictly positive on the upper half-line and strictly negative on the lower one. 

Since $|\Re\rho|\gg {1\over\nu}$, the relation (\ref{e:log_f}) implies that 
$\Re\rho\mapsto \Re F_{\nu,k}(\rho)$ defines an increasing function on each of the above half-lines. Therefore,
in order to obtain the lemma,  it suffices to check that $\Re F_{\nu,k}(\rho)>0$ on the vertical part of $b\Omegabf_c^\nu(k)$
which is contained in the line
$$\Re\rho= -{\log(c\nu)\over 2\nu}\cdot$$

Assume that $\rho$ satisfies the last identity. When $|\rho|>1$, $|1-z|$ is bounded below by a positive constant. Therefore, $\log|f(\rho)|$ is bounded below and $\Re F_{\nu,k}(\rho)>0$ for $\nu$ large enough. Otherwise, we deduce from (\ref{e:rho}) that $|f(\rho)|\gtrsim |\rho|^{2/3}$. Therefore, 
since $|\Re\rho| \gg {1\over \nu}$, we have 
$$\log|f(\rho)| \geq \log |\rho|^{2/3} +\const \geq -{2\over 3}\log\nu +\const.$$
It follows that $\Re F_{\nu,k}(\rho)$ is strictly positive when $\nu$ is large enough. This completes the proof of the lemma. 
\endproof

Denote by $z_{\nu,k}$ the complex number in $\C_+$ such that $\rho(z_{\nu,k})=\rho_{\nu,k}$. 
We have the following lemma.

\begin{lemma} \label{l:sol_positive}
Assume that $\nu$ is large enough.
Then, for every $k\geq -{\nu\over 2}+2$, we have $\Re z_{\nu,k}>0$ and $|z_{\nu,k}|>{1\over 2}$. Moreover, if $\rho=\rho(z)$ in $\Omegabf^\nu_c$ is a zero of $g_\nu(z)$ with $\Re z\geq 0$ and $\Im z\geq 0$, then $z=z_{\nu,k}$ and $\rho=\rho_{\nu,k}$ for some $k\geq-{\nu\over 2}-2$. 
\end{lemma}
\proof
When  $k\geq -{\nu\over 2}+2$, since $\rho_{\nu,k}$ belongs to $\Omegabf^\nu_c(k)$, we have $\Im\rho_{\nu,k}>-{\pi\over 2}$. Hence, $\Re z_{\nu,k}>0$. 
Since $\Re\rho_{\nu,k}<0$, we have  $z_{\nu,k}\not\in\K_+$ and hence,
$|z_{k,\nu}|>{1\over 2}$. 
If $\rho$ and $z$ are as in the lemma, then $\Im\rho\geq -{\pi\over 2}$. Such a point $\rho$ should be in $\overline\Omegabf_c^\nu(k)$ for some $k\geq -{\nu\over 2}-2$. Lemma \ref{l:sol_unique} implies the result.
\endproof

For $l\in\N$ define $\nu:=l+{d\over 2}-1$ (we use here the notation of Stefanov \cite{Stefanov} which is slightly different from the one by Zworski \cite[(25)]{Zworski2}). 
For $r>0$, denote by $n^+_l(r)$ (resp. $n^-_l(r)$) the number of points $z_{\nu,k}$ in  $\overline \D({r\over \nu})$  with $k>0$ (resp. $-{\nu\over 2}+2<k \leq 0$).
Theorem \ref{t:main_1} is a consequence of the following two propositions whose proofs will be given at the end of the section.

\begin{proposition} \label{p:number_sol_1}
Assume that $\nu$ and $r$ are large enough. Then,
\begin{enumerate}
\item[{\rm (a)}] $n^+_l(r)=0$ for $l\geq 2r$;
\item[{\rm (b)}] $z_{\nu,k}\not\in \overline\D({r\over \nu})$ if $k>cr$ for a fixed constant $c>0$ large enough; in particular, we have 
$n_l^+(r)\leq cr$;
\item[{\rm (c)}]
We have for every constant $\epsilon>0$
$$\sum_{0\leq l\leq 2r} n^+_l(r)(\dim H_l) = {\volume(\B)^2\over (2\pi)^d}r^d + O(r^{d-3/4+\epsilon}) \quad \mbox{as} \quad r\to\infty.$$
\end{enumerate}
\end{proposition}

\begin{proposition} \label{p:number_sol_2}
Assume that $\nu$ and $r$ are large enough. Then,
$n^-_l(r)=0$ for $l\geq 2r$. 
Moreover, we have for every constant $\epsilon>0$
$$\sum_{0\leq l\leq 2r} n^-_l(r)(\dim H_l) ={r^d\over \pi d(d-2)!} \int_{\partial\K_+}{|1-z^2|^{1/2}\over |z|^{d+1}}|dz| +O(r^{d-3/4+\epsilon}) \quad \mbox{as} \quad r\to\infty.$$
\end{proposition}

\medskip

\noindent
{\bf End of the proof of Theorem \ref{t:main_1}.} 
Using the decomposition of functions into spherical harmonics, Zworski relates scattering poles of $-\Delta+V$  to the zeros of a sequence of functions of the form (\ref{e:pole}) with $l\in\N$. More precisely, 
if $\rho=\rho(z)$ is a zero of $g_\nu(z)$ with $z\in\J^\nu_c$, then $-\nu z$ is a scattering pole 
with multiplicity $\dim H_l$, see \cite[\S 2]{Zworski1}. If 
$n_l(r)$ is the number of zeros of $g_\nu(z)$ in $\J^\nu_c\cap \overline\D({r\over\nu})$ with unknown $z$, Zworski proved that $n_l(r)\lesssim r$ and $n_l(r)=0$ for $l>2r$, see \cite[p.386]{Zworski2}. Therefore,
we have 
\begin{equation*} \label{e:number_scattering}
n_V(r)=\sum_{l\geq 0} n_l(r) (\dim H_l)=\sum_{0\leq l\leq 2r} n_l(r) (\dim H_l).
\end{equation*}
By Lemma  \ref{l:H_l}, we only need to consider $l$ large enough.

In our setting with a real potential, the scattering poles are symmetric with respect to the real line $\Re z=0$. 
This and Lemma \ref{l:sol_positive} imply that
\begin{equation*} \label{e:sym_scattering}
2n_l^+(r)+2n_l^-(r)\leq n_l(r)\leq 2n_l^+(r)+2n_l^-(r)+4.
\end{equation*}
Now, in order to obtain the result, it suffices to apply
 Propositions \ref{p:number_sol_1}, \ref{p:number_sol_2}  and the identity (\ref{e:c_d}).
\hfill $\square$

\bigskip

We give now the proofs of the above propositions. 
For $k\in\Z\setminus\{0\}$ such that  $k\geq -{\nu\over 2}+2$, define
$$\rho^\star_{\nu,k}:=-{\log\sigma\over 2\nu}- {\log\nu\over\nu} -{1\over 2\nu} \log f\big({k\pi i\over \nu}\big) +{k\pi i\over \nu}\cdot$$ 
These points are easier to count and we will compare them with $\rho_{\nu,k}$. 

\begin{lemma} \label{l:rho_star}
Assume that $\nu$ is large enough, $k\geq -{\nu\over 2}+2$  and $|k|\geq \nu^{1/4}$. 
Then, we have 
$$|F_{\nu,k}(\rho_{\nu,k}^\star)|\ll \nu^{-6/5}
\ \mbox{ and } \ 
|h_\nu(\rho_{\nu,k}^\star)|\ll \nu^{-1/5}.$$
\end{lemma}
\proof
Observe that 
$$|h_\nu(\rho_{\nu,k}^\star)|=|\sigma| \Big|e^{-2\nu F_{\nu,k}(\rho_{\nu,k}^\star)}-1\Big|.$$
So the second inequality in the lemma is a consequence of the first one. We prove now the first inequality. It is enough to check that 
$$\Big|\log f(\rho^\star_{\nu,k})-\log f({k\pi i\over \nu})\Big|\ll \nu^{-1/5}.$$

Consider first the case where $|k|\leq \nu$. In this case, $\rho_{\nu,k}^\star$ and ${k\pi i\over\nu}$ are bounded.
Since $|k|\geq \nu^{1/4}$, 
we have $|\rho|\gtrsim \nu^{-3/4}$ for $\rho$ in $\Omegabf_c^\nu(k)$. 
This together with  (\ref{e:log_f}) implies that
$$\Big|{\partial\log f\over \partial\rho}\Big|\lesssim 1+ {1\over |\rho|}\lesssim \nu^{3/4}.$$

Using the last inequality and the estimate
$$\Big|\rho^\star_{\nu,k}-{k\pi i\over \nu}\Big| \lesssim {\log\nu\over \nu},$$
we obtain that 
$$\Big|\log f(\rho^\star_{\nu,k})-\log f({k\pi i\over \nu})\Big|
\lesssim  \nu^{3/4}{\log\nu\over \nu} \ll \nu^{-1/5}.$$

When $k\geq\nu$  the segment joining $\rho^\star_{\nu,k}$ and $ {k\pi i\over \nu}$ is contained in the half-plane $\Im\rho>{\pi\over 2}$. Therefore, by (\ref{e:f_log_f}), we have on this segment
$$\Big|{\partial\log f\over \partial\rho}\Big| \lesssim {1\over |z|} \lesssim {1\over |\rho|}\lesssim {\nu\over k}\cdot$$
Hence, using that 
$$\Big|\rho^\star_{\nu,k}-{k\pi i\over \nu}\Big| \lesssim {\log k\over \nu},$$
we obtain again
$$\Big|\log f(\rho^\star_{\nu,k})-\log f\big({k\pi i\over \nu}\big)\Big|
\lesssim  {\nu\over k} {\log k\over \nu}
\ll \nu^{-1/5}.$$
This completes the proof of the lemma.
\endproof

\begin{lemma} \label{l:sol_app}
Assume that $\nu$ is large enough,  $k\geq -{\nu\over 2}+2$ and $|k|\geq \nu^{1/4}$. Then, we have 
$$|\rho_{\nu,k}-\rho^\star_{\nu,k}| < {2\pi\over \nu}\cdot$$
\end{lemma}
\proof
Observe that $|\Im\rho_{k,\nu}^\star|\gg {1\over \nu}$.  Using (\ref{e:rho}), we deduce that 
$\log \big|f\big({k\pi i\over \nu}\big)\big|\geq -{2\over 3}\log\nu$. Hence,
$$\Re\rho_{\nu,k}^\star\leq - {2\log\nu +\const\over 3\nu}\cdot$$
Therefore, the vertical lines of equations
$$\Re\rho -\Re\rho_{\nu,k}^\star=\pm{\pi\over 2\nu}$$
 are contained in $\Omegabf_c^\nu$.

Let $Q$ denote the square of size ${\pi\over \nu}\times {\pi\over \nu}$
limited by the above lines and the horizontal part of $b\Omegabf_c^\nu(k)$.
It is enough to show that $\rho_{\nu,k}$ belongs to $Q$. Arguing as in Lemmas \ref{l:h_function}, \ref{l:g_h} and \ref{l:sol_unique}, we only have to check that $|h_\nu(z)|$ is bounded below by a positive constant when $\rho$ is on the vertical part of $bQ$ and that $\Re F_{\nu,k}(\rho)$ is positive (resp. negative) on the right (resp. left) vertical part of $bQ$. 

We only consider the case where $\rho$ is on the left vertical part of $bQ$. The other case can be obtained in the same way.  We have
\begin{equation} \label{e:Re_rho}
\Re\rho -\Re\rho_{\nu,k}^\star=-{\pi\over 2\nu}\cdot
\end{equation}
So $|\rho-\rho_{\nu,k}^\star|\lesssim {1\over \nu} \ll |\rho_{\nu,k}^\star|$. 
As in Lemma \ref{l:rho_star}, we obtain that
\begin{equation} \label{e:f_rho}
|\log f(\rho)-\log f(\rho_{\nu,k}^\star)|\ll \nu^{-1/5}.
\end{equation}

It follows from (\ref{e:Re_rho}) and (\ref{e:f_rho}) that
$${h_\nu(\rho)+\sigma\over h_\nu(\rho_{\nu,k}^\star)+\sigma}=e^\pi+O(\nu^{-1/5}).$$
By Lemma \ref{l:rho_star}, $h_\nu(\rho_{\nu,k}^\star)$ is small. Therefore, the last identity implies that 
$|h_\nu(\rho)|$ is bounded below by a positive constant.

We also deduce from (\ref{e:Re_rho}) and the definition of $F_{\nu,k}(\rho)$ that
$$\Re F_{\nu,k}(\rho)-\Re F_{\nu,k}(\rho_{\nu,k}^\star)\leq -{\pi\over 2\nu} +{1\over 2\nu}|\log f(\rho)-\log f(\rho_{\nu,k}^\star)|=-{\pi\over 2\nu} + O(\nu^{-6/5}).$$
This and Lemma \ref{l:rho_star} imply that $\Re F_{\nu,k}(\rho)<0$. 
The result follows.
\endproof

\begin{lemma} \label{l:near_zero_Bessel}
Let $\epsilon_0$ be the constant given in Lemma  \ref{l:zero_Bessel_bis}. Let $\epsilon>0$ be any fixed constant.
Assume that $\nu$ is large enough and $\nu^{1/4}\leq k\leq \epsilon_0 \nu^4$. Then,  we have
$$|\nu z_{\nu,k}-\nu \widetilde z_{\nu,k}|\leq  |\nu \widetilde z_{\nu,k}|^{1/4+\epsilon}.$$
\end{lemma}
\proof
By Lemma \ref{l:zero_Bessel_bis}, we have
$$\Big|\widetilde\rho_{\nu,k}-\big({3\pi i\over 4\nu} +{k\pi i\over \nu}\big)\Big|\leq {1\over\nu}\cdot$$
Since $k$ is positive, $f({k\pi i\over \nu})$ is real and negative. Using Lemma \ref{l:sol_app} and the definition of $\rho_{\nu,k}^\star$, we deduce that 
\begin{equation} \label{e:rho:rho}
|\rho_{\nu,k}-\widetilde\rho_{\nu,k}|\leq | \rho_{\nu,k}-\rho^\star_{\nu,k}| + |\rho^\star_{\nu,k}-\widetilde\rho_{\nu,k}|
\lesssim {1\over \nu} + {\log k\over \nu}\lesssim {\log k\over \nu} \cdot
\end{equation}
We distinguish two cases. 

Assume first that $k$ is bounded below by a fixed small constant times $\nu$. Then,
$|\rho_{\nu,k}|$ and $|\widetilde \rho_{\nu,k}|$ are bounded below by a positive constant. It follows that
${\partial z\over\partial\rho}$ is bounded on the segment joining $\rho_{\nu,k}$ and $\widetilde \rho_{\nu,k}$. 
This and (\ref{e:rho:rho}) imply that
$$|\nu z_{\nu,k}-\nu \widetilde z_{\nu,k}|\lesssim \nu|\rho_{\nu,k}-\widetilde\rho_{\nu,k}| \lesssim \log k\lesssim \log|\nu\widetilde \rho_{k,\nu}|.$$
On the other hand, if $\rho$ is in  $i\R_+$, we have $z\in [1,+\infty)$ and
$$\rho=\Big(\sqrt{z^2-1}+\arccos {1\over z}\Big) i.$$
Therefore, $|\widetilde\rho_{\nu,k}|\lesssim |\widetilde z_{\nu,k}|$. We conclude that 
$$|\nu z_{\nu,k}-\nu \widetilde z_{\nu,k}|\lesssim \log |\nu \widetilde z_{\nu,k}|.$$

Assume now that $k$ is bounded above by a small constant times $\nu$. Then, $|\rho_{\nu,k}|$, $|\widetilde \rho_{\nu,k}|$ are small and $z_{\nu,k}$, $\widetilde z_{\nu,k}$ are close to 1 . By (\ref{e:rho}), we have $\big|{\partial z \over \partial\rho}\big| \lesssim |\rho|^{-1/3}$ for $\rho$ small. In particular, we have $\big|{\partial z \over \partial\rho}\big| \lesssim \big({\nu\over k}\big)^{1/3}$
for $\rho$ in
the segment joining $\rho_{\nu,k}$ and $\widetilde \rho_{\nu,k}$.  Hence, since $k\geq \nu^{1/4}$, we obtain
$$|\nu z_{\nu,k}-\nu \widetilde z_{\nu,k}|\lesssim \nu |\rho_{\nu,k}-\widetilde\rho_{\nu,k}|  \big({\nu\over k}\big)^{1/3}
\lesssim (\log k)({\nu\over k})^{1/3}\ll \nu^{1/4+\epsilon}.$$
This completes the proof of the lemma.
\endproof

\noindent
{\bf End of the proof of Proposition \ref{p:number_sol_1}.} 
(a) Since $\Re\rho_{\nu,k} < 0$, $z_{\nu,k}$ belongs to $\C_+\setminus \K_+$. In particular, we have $|z_{\nu,k}|>{1\over 2}$. Therefore, when $l\geq 2r$, we have $\nu\geq 2r$ and $z_{\nu,k}\not\in\overline\D({r\over\nu})$. It follows that $n^+_l(r)=0$ for $l>2r$. 

\medskip\noindent
(b) Assume that $l\leq 2r$ and $k\geq cr$ for some constant $c>0$ large enough. Then,  $\Im\rho_{\nu,k}$ is bounded below by a large positive constant times $r\over\nu$. Using the definition of $\rho$, we obtain that $|z_{\nu,k}|$ is almost equal to $|\rho_{\nu,k}|$. In particular, $z_{\nu,k}$ does not belong to $\overline\D({r\over\nu})$. It follows that $n^+_l(r)\leq cr$. 

\medskip\noindent
(c) 
Observe that by (b) and Proposition \ref{p:Weyl}, 
if $l\geq L_r:=2(c\epsilon_0^{-1}r)^{1/4}$ and $k\geq \epsilon_0\nu^4$, then $z_{\nu,k}$ and $\widetilde z_{\nu,k}$ do not belong to $\overline\D({2r\over\nu})$. Hence, by  Lemma \ref{l:near_zero_Bessel},
we have for  $l\geq L_r$
$$n_l^+(r)\geq m_l^+(r-r^{1/4+\epsilon}) -\nu^{1/4} = m_l^+(r-r^{1/4+\epsilon})+O(l^{1/4}).$$
This, together with Lemma \ref{l:H_l} and Proposition \ref{p:Weyl}, implies that
\begin{eqnarray*}
\sum_{0\leq l\leq 2r} n_l^+(r)(\dim H_l) & \geq & \sum_{L_r\leq l\leq 2r} n_l^+(r)(\dim H_l) \\
& \geq & \sum_{L_r\leq l\leq 2r} \big[m_l^+(r-r^{1/4+\epsilon})+O(l^{1/4})\big](\dim H_l) \\
& \geq & {\volume(\B)^2\over (2\pi)^d}r^d + O(r^{d-{3/4}+\epsilon}).
\end{eqnarray*}

For the converse estimate, in the same way, we have for  $l\geq L_r$
$$n_l^+(r)\leq m_l^+(r+r^{1/4+\epsilon}) +\nu^{1/4} = m_l^+(r+r^{1/4+\epsilon})+O(l^{1/4}).$$
Using parts (a) and (b), we obtain
\begin{eqnarray*}
\sum_{0\leq l\leq 2r} n_l^+(r)(\dim H_l) & = & \sum_{L_r\leq l\leq 2r} n_l^+(r)(\dim H_l)  + O(r^{d/4+{3/4}}) \\
& \leq & \sum_{0\leq l\leq 2r} \big[m_l^+(r+r^{1/4+\epsilon})+O(l^{1/4})\big](\dim H_l) + O(r^{d/4+{3/4}}) \\
& \leq & {\volume(\B)^2\over (2\pi)^d}r^d + O(r^{d-{3/4}+\epsilon}).
\end{eqnarray*}
This completes the proof of the proposition.
\hfill $\square$

\bigskip

Define for  $-{\nu\over 2}-2<k\leq 0$
$$\widehat z_{\nu,k}:=\rho^{-1}\big({k \pi i\over\nu}\big).$$
This point belongs to $\partial\K_+$ and $\Re \widehat z_{\nu,k}\geq 0$ if and only if $k\geq -{\nu\over 2}$.  
We have the following lemma.

\begin{lemma} \label{l:count_negative}
Let $\epsilon>0$ be a fixed constant. Assume that $\nu$ is large enough and
$-{\nu\over 2}+2<k\leq -\nu^{1/4}$. Then, we have 
$$|\nu z_{\nu,k}- \nu \widehat z_{\nu, k}|\leq |\nu \widehat z_{\nu,k}|^{1/4+\epsilon}.$$
\end{lemma}
\proof
Observe that  $\big|{k\pi i\over \nu}\big|$ is bounded by ${\pi\over 2}$. 
By Lemma \ref{l:sol_app},  $|\rho_{\nu,k}|$ is also bounded by ${\pi\over 2}$ plus a small constant. 
Hence, by (\ref{e:rho}), we have for $\rho$ in the segment joining $\rho_{\nu,k}$ and ${k\pi i\over \nu}$
$$\Big|{\partial z\over \partial \rho}\Big| \lesssim |\rho|^{-1/3}\lesssim \Big({\nu\over |k|}\Big)^{1/3}.$$ 
Moreover,
$$\Big|\rho_{\nu,k}-{k\pi i\over \nu}\Big| \lesssim {\log \nu\over \nu}\cdot$$ 
It follows that 
$$|\nu z_{\nu,k}- \nu \widehat z_{\nu, k}|\leq (\log\nu) \Big({\nu\over |k|}\Big)^{1/3}\ll \nu^{1/4+\epsilon}\lesssim (\nu\widehat z_{\nu,k})^{1/4+\epsilon}$$
since $\widehat z_{\nu,k}\in \partial\K_+$.  The lemma follows. 
\endproof

\noindent
{\bf End of the proof of Proposition \ref{p:number_sol_2}.} 
The fact that $n^-_l(r)=0$ for $l>2r$ is obtained as in Proposition \ref{p:number_sol_1}. 
Note that by definition, we have $n^-_l(r)\leq\nu\lesssim l$. 
We prove now the second assertion in the proposition.

Denote by $m^-_l(r)$ the number of integers $k$ such that $-{\nu\over 2}+2\leq k\leq 0$ and $|\widehat z_{\nu,k}|\leq {r\over \nu}$. Since $\widehat z_{\nu,k}\in \partial \K_+$, we have $m^-_l(r)=0$ when $l>2r$.
Observe that $\rho$ defines a diffeomorphism between $\partial\K_+$ and $i[0,\pi]$. 
So if we define 
$$\Gamma_{\nu,r}:=\Big\{z\in\partial\K_+:\  |z|\leq {r\over \nu}, \Re z>0,\Im z>0\Big\}$$
then we have
$$\Big|m_l^-(r)  -  {\nu\over \pi} \length(\rho(\Gamma_{\nu,r}))\Big| \leq 3.$$
Thus,
\begin{equation} \label{e:m_negative}
\Big|m_l^-(r)- {\nu\over \pi} \int_{\Gamma_{\nu,r}} id\rho\Big|\leq 3.
\end{equation}

The last inequality implies that $m^-_l(r)\lesssim\nu\lesssim l$. Observe also that when $l>2r$ we have $\Gamma_{\nu,r}=\varnothing$.
Define $\Gamma:=\{z\in \partial \K_+:\ \Re(z)>0\}$ and $r':=r-r^{1/4+\epsilon}$. 
By Lemma \ref{l:count_negative}, we have 
$$n^-_l(r)\geq m^-_l(r')-\nu^{1/4} = m^-_l(r')+ O(l^{1/4}).$$
We deduce from the above discussion that 
\begin{eqnarray*}
\sum_{0\leq l\leq 2r} n^-_l(r) (\dim H_l) 
& \geq &  \sum_{0\leq l\leq 2r} \big(m^-_l(r')+O(l^{1/4})\big) (\dim H_l)\\
& = &  \sum_{l\geq 0} m^-_l(r')(\dim H_l)+O(r^{d-3/4}).
\end{eqnarray*}

By (\ref{e:m_negative}), the last sum is equal to
\begin{eqnarray*}
\sum_{l\geq 0} {2l^{d-1}\over \pi (d-2)!} \int_{\Gamma_{\nu,r'}} id\rho + O(r^{d-3/4}) 
& = &  {2\over \pi (d-2)!}\int_\Gamma  id\rho \sum_{\nu\leq {r'\over |z|}} l^{d-1}+ O(r^{d-3/4})\\
& = &  {1\over \pi (d-2)!}\int_{\partial\K_+}  id\rho \sum_{\nu\leq {r'\over |z|}} l^{d-1}+ O(r^{d-3/4})\\
& = &  {r'^d+O(r'^{d-1})\over \pi d(d-2)!}\int_{\partial\K_+}  {id\rho\over |z|^d} + O(r^{d-3/4})\\
& = &  {r^d\over \pi d(d-2)!}\int_{\partial\K_+}  {id\rho\over |z|^d} + O(r^{d-3/4+\epsilon}).
\end{eqnarray*}
We conclude that 
$$\sum_{0\leq l\leq 2r} n^-_l(r) (\dim H_l) \geq {r^d\over \pi d(d-2)!}\int_{\partial\K_+}  {|1-z^2|^{1/2}\over |z|^{d+1}}|dz| + O(r^{d-3/4+\epsilon}).$$

We obtain the converse inequality in the same way using that 
$$n^-_l(r)\leq m^-_l(r'')+\nu^{1/4} = m^-_l(r'')+ O(l^{1/4})$$
with $r'':=r+r^{1/4+\epsilon}$. This completes the proof of the proposition.
\hfill $\square$

%%%%%%%%%%%%

\section{Generic potentials in a holomorphic family} \label{s:holo}

In this section, we prove Theorem \ref{t:main_2}. We need the following 
result which relates the asymptotic behavior of $n_V(r)$ and of $N_V(r)$. It holds for any bounded complex potential $V$ with compact support.

\begin{proposition} \label{p:n:N}
Let $\delta$ and $A$ be strictly positive constants such that $\delta<d$. Then, we have $(a)\Rightarrow (b)\Rightarrow (c)$, where
\begin{enumerate}
\item[]{ \rm (a)}   $\qquad n_V(r)=A r^d+ O(r^{d-\delta})\quad \mbox{as}\quad r\to\infty$;
\item[]{ \rm (b)}   $\qquad N_V(r)= \frac{Ar^d}{d} + O(r^{d-\delta})\quad \mbox{as}\quad r\to \infty$;
\item[]{ \rm (c)}   $\qquad n_V(r)=A r^d+ O(r^{d-\frac{\delta}{2}})\quad \mbox{as}\quad r\to\infty$.
\end{enumerate}
\end{proposition}    
\begin{proof}
The proof is similar to Lemma 1 in \cite{Stefanov}. 
The implication $(a)\Rightarrow (b)$ follows from the definition of $N_V(r)$. 
Assume that the property (b) holds. We show that (c) is also true. We only have to consider $r$ large enough.

Choose a constant $c>0$ such that $|dN_V(r)-Ar^d|\le  c r^{d-\delta}$.  Define $\alpha:=c r^{1-{\delta\over 2}}.$ Since $n_V(r)$ is increasing in $r$, we have for $r$ large enough
$$ \frac{n_V(r)}{r+\alpha}\le  {1\over \alpha}\int_{r}^{r+\alpha}\frac{n_V(t)}{t}dt \leq {N(r+\alpha)-N(r) \over \alpha}+{n_V(0)\over r} \cdot$$ 
It follows from the above estimate on $dN_V(r)-Ar^d$ that 
\begin{eqnarray*}
{n_V(r)\over r+\alpha} & \leq & {A\over d}\big[(r+\alpha)^d-r^d\big]+{2c(r+\alpha)^{d-\delta}\over d\alpha} +{n_V(0)\over r}\\
& \leq &  A(r+\alpha)^{d-1}+{2c(r+\alpha)^{d-\delta}\over d\alpha} +{n_V(0)\over r}\\
& \leq &  A(r+\alpha)^{d-1}+O(r^{d-1-{\delta\over 2}}).
\end{eqnarray*}
Hence, 
$$n_V(r)\leq A(r+\alpha)^{d}+O(r^{d-{\delta\over 2}})\leq Ar^d+O(r^{d-{\delta\over 2}}).$$

In the same way, using the inequalities
$${N(r)-N(r-\alpha) \over \alpha}
\le  {1\over \alpha}\int_{r-\alpha}^{r}\frac{n_V(t)}{t}dt \le \frac{n_V(r)}{r-\alpha}$$
we obtain
$$n_V(r)\geq Ar^d+O(r^{d-{\delta\over 2}}).$$
The proposition follows.
\end{proof}

Following Christiansen \cite{Christiansen4}, we will reduce the problem to the study of a family of plurisubharmonic (p.s.h. for short) functions. The reader will find in 
Demailly \cite{Demailly} and Lelong-Gruman \cite{LelongGruman}  basic properties of p.s.h. functions.

Recall that a function $\Phi:\Omega\to\R\cup\{-\infty\}$ is p.s.h. if it is not identically equal to $-\infty$ and if its restriction to each holomorphic disc is either subharmonic or equal to $-\infty$. A subset of $\Omega$ is pluripolar if it is contained in the pole set $\{\Phi=-\infty\}$ of a p.s.h. function $\Phi$.
The following lemma is crucial for the proof of Theorem \ref{t:main_2}.

\begin{lemma} \label{l:psh}
Let $\Phi_n$, $n=1,2,\ldots$, be a sequence of p.s.h. functions on a domain $\Omega$ of $\C^p$. Assume there are constants $c>0$ and $\gamma>1$ such that $\Phi_n\leq cn^{-\gamma}$ on $\Omega$ and  $\Phi_n(\vartheta_0)>-cn^{-\gamma}$ for some point $\vartheta_0\in\Omega$. Then for every $\alpha<\gamma-1$ there exists a pluripolar set $E\subset\Omega$ such that $\Phi_n(\vartheta)=o(n^{-\alpha})$ for every $\vartheta\in \Omega\setminus E$.
\end{lemma}
\proof
Replacing $\Phi_n$ by $\Phi_n-cn^{-\gamma}$ allows us to assume that $\Phi_n\leq 0$. Define
$$\Phi_{[m]}:=\sum_{n=1}^m n^\alpha \Phi_n \quad \mbox{and} \quad \Phi:=\sum_{n=1}^\infty n^\alpha\Phi_n.$$
So the sequence $\Phi_{[m]}$ decreases to $\Phi$. It is clear that $\Phi(\vartheta_0)\not=-\infty$. So $\Phi$ is a p.s.h. function. Therefore, $E:=\Phi^{-1}(-\infty)$ is a pluripolar set. For $\vartheta\not\in E$, we have $n^\alpha\Phi_n(\vartheta)\to 0$. The lemma follows.
\endproof

For $0<\delta \leq 1$, define
$$\widetilde{\mathfrak{M}}^\delta_a:=\Big\{V\in L^\infty(\B_a,\C):\  dN_V(r)-c_da^dr^d=O(r^{d-\delta+\epsilon}) \mbox{ as } r\to\infty \mbox{ for every } \epsilon>0\Big\}.$$
Proposition \ref{p:n:N} implies that
$$\mathfrak{M}^{\delta}_a\subset \widetilde{\mathfrak{M}}^\delta_a\subset \mathfrak{M}^{\delta/2}_a.$$
So Theorem \ref{t:main_2} is a consequence of the following result.

\begin{theorem} \label{t:main_2_bis}
Let $\Omega$ be a connected open set of $\C^p$. 
Let $V_\vartheta$ be a uniformly bounded family of potentials in $L^\infty(\B_a,\C)$ depending holomorphically on the parameter $\vartheta\in\Omega$. Suppose there are $\vartheta_0\in\Omega$ and $0<\delta\leq 1$ such that $V_{\vartheta_0}\in\widetilde{\mathfrak{M}}^\delta_a$. Then  there is a pluripolar set $E\subset \Omega$ such that $V_\vartheta\in \widetilde{\mathfrak{M}}^{\delta/2}_a$ for all $\vartheta\in \Omega\setminus E$.
\end{theorem}
\proof 
Define for $r$ large enough
$$\Psi(r,\vartheta):={1\over 2\pi r^d} \int_0^{2\pi} \log |s_{V_\vartheta}(r e^{i\theta})| d\theta-{c_d a^d\over d} \cdot$$
Since $V_\vartheta$ depends holomorphically on $\vartheta$, the function $s_{V_\vartheta}$ depends also holomorphically on $\vartheta$. Hence, $\Psi(r,\vartheta)$ is p.s.h. on $\vartheta$. 

Fix a small constant $\epsilon>0$. 
By Theorem \ref{t:upper_bound} and Proposition \ref{p:Christiansen}, we have
$$\Psi(r,\vartheta)\lesssim {\log r \over r} \quad \mbox{and} \quad \Psi(r,\vartheta_0)\gtrsim -r^{-\delta+\epsilon}.$$
Define for $k:={2\over\delta}$
$$\Phi_n(\vartheta):= \Psi(n^k,\vartheta).$$
Then
$$\Phi_n(\vartheta)\lesssim (\log n)n^{-k} \quad \mbox{and} \quad \Phi_n(\vartheta_0)\gtrsim n^{-k(\delta-\epsilon)}.$$

Lemma \ref{l:psh} implies that for $\vartheta$ outside a pluripolar set $E_\epsilon$ we have 
$$\Phi_n(\vartheta)=o(n^{-k(\delta-\epsilon)+1+\epsilon})=o(n^{-1+{2\epsilon\over\delta}+\epsilon}).$$
Hence, by Proposition \ref{p:Christiansen}, we have  for $r=n^k$ and for $\vartheta\not\in E_\epsilon$
\begin{equation}\label{e:N:speed}
N_{V_\vartheta}(r)={c_d a^d r^d\over d} + O(r^{d-{\delta\over 2}+\epsilon+{\epsilon\delta\over 2}}).
\end{equation}
Since $N_{V_\vartheta}(r)$ is increasing in $r$, we deduce that if $n^k\leq r\leq (n+1)^k$ 
\begin{eqnarray*}
\Big|N_{V_\vartheta}(r)-{c_d a^d r^d\over d}\Big| & \lesssim &  (n+1)^{kd}-n^{kd}+O(r^{d-{\delta\over 2}+\epsilon+{\epsilon\delta\over 2}}) \\
& \lesssim & n^{kd-1}+O(r^{d-{\delta\over 2}+\epsilon+{\epsilon\delta\over 2}}) \\
& \lesssim & O(r^{d-{\delta\over 2}+\epsilon+{\epsilon\delta\over 2}}).
\end{eqnarray*}
So the property (\ref{e:N:speed}) holds for 
$r\to \infty$ with $r\in\R_+$.

Finally, define $E:=\cup_{n=1}^\infty E_{1/n}$. This is a pluripolar set, see e.g. \cite{LelongGruman}. We have for all $\vartheta\not\in E$ and $\epsilon>0$
$$N_{V_\vartheta}(r)={c_d a^d r^d\over d} + O(r^{d-{\delta\over 2}+\epsilon})\quad \mbox{as} \quad r\to\infty.$$
This completes the proof of the theorem.
\endproof

\noindent
T.-C. Dinh, UPMC Univ Paris 06, UMR 7586, Institut de
Math{\'e}matiques de Jussieu, 4 place Jussieu, F-75005 Paris, France.\\ 
{\tt dinh@math.jussieu.fr}, {\tt http://www.math.jussieu.fr/$\sim$dinh}

\medskip

\noindent
D.-V.  Vu,
Department of Mathematics,
Hanoi National University of Education,
136 Xuan Thuy str., Hanoi, Vietnam.\\
{\tt vuvietsp@gmail.com}

\end{document}